\documentclass{amsart}
\usepackage{amssymb,url,xspace, amsthm}
\usepackage{bigints}
\usepackage{multirow}
\usepackage[all]{xy}
\usepackage{hyperref}

\newtheorem{theorem}{Theorem}[section]
\newtheorem{lemma}[theorem]{Lemma}
\newtheorem{props}[theorem]{Proposition}

\theoremstyle{definition}
\newtheorem{definition}[theorem]{Definition}

\newtheorem*{question}{Question}
\newtheorem*{theoremtype}{Theorem 2.14}
\newtheorem*{theoremcotype}{Theorem 3.4}
\newtheorem*{rmk}{Remark}

\numberwithin{equation}{section}

\DeclareMathOperator*{\esssup}{ess\,sup}

\newcommand{\vertiii}[1]{{\left\vert\kern-0.25ex\left\vert\kern-0.25ex\left\vert #1 
    \right\vert\kern-0.25ex\right\vert\kern-0.25ex\right\vert}}

\def\block(#1,#2)#3{\multicolumn{#2}{c}{\multirow{#1}{*}{$ #3 $}}}

\begin{document}

\title{Type, cotype and twisted sums induced by complex interpolation}

\author{Willian Hans Goes Corr\^ea}
\address{Departamento de Matem\'atica, Instituto de Matem\'atica e Estat\'istica, Universidade de S\~ao Paulo, Rua do Mat\~ao 1010, 05508-090 S\~ao Paulo SP, Brazil}
\thanks{The present work was produced with support from CNPq, National Council for Scientific and
Technological Development - Brazil, grant 140413/2016-2,
and CAPES, Coordena\c{c}\~ao de Aperfei\c{c}oamento de Pessoal de N\'ivel Superior, grant 1328372}

\date{}

\begin{abstract}
This paper deals with extensions or twisted sums of Banach spaces that come induced by complex interpolation and the relation between the type and cotype of the spaces in the interpolation scale and the nontriviality and singularity of the induced extension. The results are presented in the context of interpolation of families of Banach spaces, and are applied to the study of submodules of Schatten classes. We also obtain nontrivial extensions of spaces without the CAP which also fail the CAP.
\end{abstract}

\maketitle

\section{Introduction}
A twisted sum of Banach spaces $Y$ and $Z$ is a short exact sequence
  \[
  \xymatrix{ 0 \ar[r] & Y \ar[r]^{i} & X
    \ar[r]^{q} & Z\ar[r]  &0}
  \]

\noindent where $X$ is a quasi-Banach space and the arrows are bounded linear maps.

The theory of twisted sums of Banach spaces has been successfully used in the study of the so called 3-space properties. Given a property $P$ of Banach spaces, $P$ is said to be a 3-space property (3SP) if for every Banach space $X$ having a closed subspace with $P$ and respective quotient also with $P$, $X$ has $P$. Simple examples of 3SP are separability, having finite dimension and reflexivity.

To see how twisted sums and 3SP are related, consider the following problem of Palais: is being isomorphic to a Hilbert space a 3SP? Finding a counterexample to this problem corresponds to finding a twisted sum
  \[
  \xymatrix{ 0 \ar[r] & H_1 \ar[r] & X
    \ar[r] & H_2\ar[r]  &0}
  \]
  
\noindent where $H_1$ and $H_2$ are a Hilbert spaces and $H_1$ is not complemented in $X$ (in this case, $X$ is automatically isomorphic to a Banach space).

In \cite{Enflo01}, Enflo, Lindenstrauss and Pisier show that there is such counterexample. Kalton and Peck gave a solution some years later (\cite{Kalton01}) using nonlinear maps called quasi-linear. They obtained a twisted sum:
\begin{equation}\label{eq:1.1}
  \xymatrix{ 0 \ar[r] & \ell_2 \ar[r] & Z_2
    \ar[r] & \ell_2\ar[r]  &0}
\end{equation}
  
\noindent in which the copy of $\ell_2$ in $Z_2$ is not complemented.

The Kalton-Peck space $Z_2$ also appears in a construction due to Rochberg and Weiss which is possible every time we have a compatible pair $(X_0, X_1)$ in the sense of interpolation (see \cite{Rochberg01}). Given $(X_0, X_1)$, for each $\theta \in (0, 1)$ we have a twisted sum of $X_{\theta}$ with itself, denoted $dX_{\theta}$. In this context the Kalton-Peck space comes induced by the interpolation scale $(\ell_{\infty}, \ell_1)$ at $\theta = \frac{1}{2}$.

Given two Banach spaces $Y$ and $Z$, we always have a trivial twisted sum by means of the direct sum $Y \oplus Z$ with the obvious inclusion and quotient map. So the first objective in the theory is obtaining nontrivial twisted sums, i.e., twisted sums in which the subspace is not complemented in the middle space.

In \cite{Kalton03}, under the assumption of super-reflexivity, Kalton shows that the twisted sum induced by a compatible pair of K\"othe function spaces $(X_0,X_1)$ is boundedly trivial (a subclass of
trivial twisted sums) exactly when $X_0=X_1$. One of the results of a preprint of Castillo, Ferenczi and Gonz\'alez (\cite{CastilloUnpublished01}) is that, under the same assumptions, the twisted sum induced by a compatible pair of K\"othe function spaces $(X_0, X_1)$ at $\theta$ is trivial precisely when $X_1$ is a weighted version of $X_0$. The question of which interpolation scales generate trivial twisted sums is still open in the general scenario.

Banach's Hyperplane Problem asks if Banach spaces are always isomorphic to its hyperplanes. In \cite{Gowers1994}, Gowers answered this problem in the negative. The Kalton-Peck space is naturally isomorphic to its subspaces of codimension $2$, but it is still unknown if it is isomorphic to its hyperplanes. Since it appears more naturally than Gower's contruction, it would be interesting to know if it also answers in the negative the Hyperplane Problem.

In that direction, \cite{Castillo2015} suggests that this problem is related to the singularity of the quotient map in the twisted sum \eqref{eq:1.1}. So another objective is to obtain twisted sums in which the quotient map is strictly singular.

Castillo, Ferenczi and Gonz\'{a}lez (\cite{Castillo01}) studied the singularity of twisted sums induced by interpolation when the spaces in the scale have a K\"othe function space or unconditional structure, or at least a Schauder basis with some form of local unconditionality. By studying the quasi-linear maps that define the twisted sums and how they relate to the structure of the spaces in the interpolation scale, they were able to obtain results ensuring nontriviality or singularity.

In this work we study twisted sums induced by interpolation when the spaces in the interpolation scale do not necessarily satisfy any form of unconditionality. We do not need the presence of a Schauder basis either. Using the classical concepts of Rademacher type and cotype of a Banach space, we give results in the same spirit of those in \cite{Castillo01}, without conditions of an unconditional or function space structure.

The structure of the paper is the following: in the remainder of the Introduction we give the necessary background on twisted sums \hyperref[sec:1.1]{(1.1)} and complex interpolation \hyperref[sec:1.2]{(1.2)}, and how the two of them are related \hyperref[sec:1.3]{(1.3)}. We do this in the context of interpolation of families, mixing the interpolation methods of \cite{Coifman1982} and \cite{Krein1982}, detailing the conditions that guarantee that the scale induces a twisted sum by means of the pushout. We point out that according to \cite{Kalton20031131} the framework of analytic families may be applied to the interpolation method of \cite{Coifman1982}, but there is one technical detail that is overlooked (see \hyperref[pro:1.5]{Proposition 1.5}).

\hyperref[sec:2]{Section 2} is devoted to obtaining conditions on the Rademacher type of the spaces of the interpolation scheme that ensure nontriviality or singularity of the induced twisted sum. Let $\mathbb{S} = \{z \in \mathbb{C} : 0 < Re(z) < 1\}$, and $\mathbb{S}_j = \{z \in \mathbb{C} : Re(z) = j\}$, $j = 0, 1$. We prove the following result:
\begin{theoremtype}
Let $\{(X_z)_{z \in \partial\mathbb{S}}, X\}$ be an interpolation family for which the spaces $X_z$, for a.e.\ $z \in \mathbb{S}_j$ have type $p_j$ with uniformly bounded constants, that $p_{X_z} = p_j$, $j=0, 1$, and that $p_0 \neq p_1$. Consider $p$ given by $\frac{1}{p} = \frac{1 - Re(z_0)}{p_0} + \frac{Re(z_0)}{p_{1}}$.
 
 \textit{a)} Suppose that $W \subset X_{z_0}$ is an infinite dimensional closed subspace such that $p_W = p$. Then the twisted sum induced by $W$ is not trivial. In particular, if $X_{z_0}$ is infinite dimensional and $p_{X_{z_0}} = p$, $dX_{z_0}$ is a nontrivial extension of $X_{z_0}$.
 
 \textit{b)} If $X_{z_0}$ is infinite dimensional and $p_W = p$ for every infinite dimensional closed subspace $W \subset X_{z_0}$, then $dX_{z_0}$ is a singular extension of $X_{z_0}$.
\end{theoremtype}

In \hyperref[sec:3]{Section 3} we use a dualization argument to prove an analogous result for nontriviality under conditions on the cotype of the spaces of the interpolation scale, this time for interpolation of couples:
\begin{theoremcotype}
Let $\overline{X} = (X_0, X_1)$ be a compatible pair of Banach spaces. Suppose $\Delta(\overline{X})$ dense in $X_0$ and in $X_1$ and that at least one of the spaces $X_0$ or $X_1$ is reflexive. Suppose also that $X_0$ and $X_1$ have type strictly bigger than $1$, that $X_0$ has cotype $q_{X_0}$, that $X_1$ has cotype $q_{X_1}$, $q_{X_0} \neq q_{X_1}$, that $X_{\theta}$ is infinite dimensional and that $q_{X_{\theta}}$ satisfies $\frac{1}{q_{X_{\theta}}} = \frac{1-\theta}{q_{X_0}} + \frac{\theta}{q_{X_1}}$. Then $d_{\Omega_{\theta}}X_{\theta}$ is a nontrivial extension of $X_{\theta}$.
\end{theoremcotype}

In \hyperref[sec:4]{Section 4}, we give some examples. We call attention to examples \ref{sec:4.3}, where we exploit the known nonsigularity of certain twisted sums to obtain results on the structure of submodules of Schatten classes, and \ref{sec:4.4}, where we get nontrivial twisted sums such that the three spaces in the short exact sequences do not have the compact approximation property. This last result seems not to be obtainable only from Kalton's results or the ones from \cite{Castillo01}.

\subsection{Twisted Sums and Extensions}\label{sec:1.1}
We present here the necessary background on the theory of twisted sums of Banach spaces.

We recall that a linear operator $T : X_1 \rightarrow X_2$ is called \textit{strictly singular} if its restriction to any closed infinite dimensional subspace of $X_1$ is not an isomorphism. If $Y$ is a closed subspace of $X_2$, we let $q_Y : X_2 \rightarrow X_2/Y$ be the quotient map. The operator $T$ it is called \textit{strictly cosingular} if $q_Y \circ T$ being surjective implies $Y$ of finite codimension in $X_2$.

Consider a twisted sum
  \[
  \xymatrix{ 0 \ar[r] & Y \ar[r]^{i} & X
    \ar[r]^{q} & Z\ar[r]  &0}
  \]
We suppose that $X$ is isomorphic to a Banach space. The twisted sum is said to be:
\begin{itemize}
    \item \textit{trivial} if $i(Y)$ is complemented in $X$;
    \item \textit{singular} if $q$ is strictly singular;
    \item \textit{cosingular} if $i$ is strictly cosingular.
\end{itemize}

One also refers to $X$ as a twisted sum of $Y$ and $Z$. If $Y = Z$, $X$ is also called an \textit{extension} of $Y$.

Given two twisted sums $X_1$ and $X_2$ of $Y$ and $Z$, they are \textit{equivalent} if there is a bounded linear operator $T : X_1 \rightarrow X_2$ making the following diagram commute:
  \[
  \xymatrix{ 0 \ar[r] & Y \ar[r]\ar@{=}[d] & X_1
    \ar[r]\ar[d]^{T} & Z\ar[r] \ar@{=}[d] &0 \\ 0 \ar[r] & Y
    \ar[r] & X_2 \ar[r]& Z\ar[r] &0}
  \]

If the identities in the diagram are instead isomorphisms, i.e., if we have a commutative diagram
  \[
  \xymatrix{ 0 \ar[r] & Y_1 \ar[r]\ar[d] & X_1
    \ar[r]\ar[d]^{T} & Z_1\ar[r] \ar[d] &0 \\ 0 \ar[r] & Y_2
    \ar[r] & X_2 \ar[r]& Z_2\ar[r] &0}
  \]
  
\noindent where the vertical arrows are isomorphisms, $X_1$ and $X_2$ are said to be \textit{isomorphically equivalent}.
  
A twisted sum of $Y$ and $Z$ is trivial if and only if it is equivalent to the trivial twisted sum $Y \oplus Z$.

Twisted sums are induced by maps $F : Z \rightarrow Y$ called \textit{quasi-linear} (\cite{Kalton01}), which are homogeneous maps for which there is a positive constant $K$ such that
\[
\|F(z_1 + z_2) - F(z_1) - F(z_2)\| \leq K(\|z_1\| + \|z_2\|)
\]

\noindent for all $z_1, z_2 \in Z$. A quasi-linear map $F$ induces a twisted sum $Y \oplus_F Z$ of $Y$ and $Z$, which is $Y \times Z$ with the quasi-norm
\[
\|(y, z)\| = \|y - Fz\|_Y + \|z\|_Z
\]

\noindent with the obvious inclusion and quotient. Reciprocally, given a twisted sum $X$ of $Y$ and $Z$, there is a quasi-linear map $F : Z \rightarrow Y$ such that $X$ is equivalent to $Y \oplus_F Z$: simply take $F = i^{-1} \circ (B - L)$, where $B$ is any bounded homogeneous selection for the quotient map, and $L$ is any linear selection for the quotient map (we recall that a selection is a right inverse).

Two quasi-linear maps $F, G : Z \rightarrow Y$ are said to be \textit{equivalent} if they induce equivalent twisted sums. In the same way, we have the concept of \textit{isomorphically equivalent} quasi-linear maps.
 
A quasi-linear map $F : Z \rightarrow Y$ is \textit{trivial} if there is a linear map $A : Z \rightarrow Y$ such that
\[
\|F - A\| = \sup\limits_{\|z\| \leq 1} \|F(z) - A(z)\| < \infty
\]

This happens if and only if $F$ induces the trivial twisted sum $Y \oplus Z$.

Given a closed subspace $W$ of $Z$, there is a twisted sum
  \[
  \xymatrix{ 0 \ar[r] & Y \ar[r] & q^{-1}(W)
    \ar[r] & W\ar[r]  &0}
  \]
  
\noindent defined by $F|_{W}$. The twisted sum $X$ is singular if and only if $F|_{W}$ is nontrivial for every closed infinite dimensional subspace $W$ of $Z$. If the context is clear, we call the previous short exact sequence the \textit{twisted sum induced by} $W$.

Given a twisted sum $X$ of $Y$ and $Z$ and a bounded linear operator $\alpha : Y \rightarrow Y'$, we have a twisted sum of $Y'$ and $Z$ by means of the pushout (\cite{Castillo2013}):
  \[
  \xymatrix{ 0 \ar[r] & Y \ar[r]\ar[d]^{\alpha} & X
    \ar[r]\ar[d] & Z\ar[r] \ar@{=}[d] &0 \\ 0 \ar[r] & Y'
    \ar[r] & PO \ar[r]& Z\ar[r] &0}
  \]

Here, $PO = (X \oplus Y')/\Delta$, where $\Delta = \{(i(y), -\alpha(y)) : y \in Y\}$, the embedding is $y' \mapsto (0, y') + \Delta$, and the quotient map is $(x, y') + \Delta \mapsto q(x)$. If $X$ is defined by the quasi-linear map $F$, then $PO$ is defined by the quasi-linear map $\alpha \circ F$. Notice that if $X$ is a Banach space, so is $PO$.

For more information on twisted sums, we refer the reader to \cite{Castillo02}.

\subsection{Complex Interpolation}\label{sec:1.2}
Now we describe the complex method of interpolation we use, which is a modification of that of \cite{Coifman1982}, mixing it with that of \cite{Krein1982} (in the sense that, as in \cite{Krein1982}, we choose the intersection space). The modification is minor, so we refer to \cite{Coifman1982} for more details.

It was somewhat predicted by the authors of the original method (see their Appendix 2), but we do not know if it was presented elsewhere. This modification is only in order to simplify the calculation of the interpolation space. For example, if one wants to obtain Ferenczi's space (\cite{Ferenczi01}), $c_{00}$ is the natural choice of intersection space, and one does not have to actually compute the intersection. 

The applications given in \hyperref[sec:4]{Section 4} use only interpolation of couples, but since the proofs of \hyperref[sec:2]{Section 2} pass without great difficulty to the context of families, we present them in this more general scenario.

We denote by $\mathbb{D}$ the open unit disk in the complex plane, and by $\mathbb{S}$ the strip $\{z \in \mathbb{C} : 0 < Re(z) < 1\}$.

If $z = j + it$, we let $dP_{z_0}(z) = dP_{z_0}(j + it) = P(z_0, j + it)dt$ be the harmonic measure on $\partial\mathbb{S}$ with respect to the point $z_0$, where $P(z_0, j + it)$ is the Poisson kernel on the strip. We denote by $\mathbb{S}_j$ the line $Re(z) = j$, $j = 0, 1$.

A family of Banach spaces $\{(X_z, \|.\|_z) : z \in \partial\mathbb{S}\}$ is an \textit{interpolation family} if each $X_z$ is continuously linearly included in a Banach space $\mathcal{U}$ and we fix a subspace $X$ of $\cap_{z \in \partial\mathbb{S}} X_z$ satisfying for some $z_0 \in \mathbb{S}$:
\begin{enumerate}
\item[I1] For every $x \in X$, $\|x\|_z$ is a measurable function on $\partial\mathbb{S}$ with respect to $dP_{z_0}$ and 
\[
\int_{\partial\mathbb{S}} \log^+\|x\|_{z} dP_{z_0}(z) < \infty
\]
\noindent where $\log^+ (t) = \max\{0, \log t\}$.
\item[I2]
There is a (fixed) measurable function $k$ on $\partial\mathbb{S}$ such that $\int\limits_{\partial\mathbb{S}} \log^+ k(z) dP_{z_0}(z) < \infty$ and for every $x \in X$ and every $z \in \partial\mathbb{S}$ we have $\|x\|_{\mathcal{U}} \leq k(z)\|x\|_z$.
\end{enumerate}

It follows that these conditions are satisfied for every $z_0 \in \mathbb{S}$. The space $\mathcal{U}$ is called a \textit{containing space} for the family.

So an interpolation family is actually a pair $\{(X_z)_{z \in \partial\mathbb{S}}, X\}$ satisfying the above conditions. The space $X$ plays a role analogous to that of the intersection space of classical complex interpolation for couples of Banach spaces. 

We recall now the definitions and some basic properties of the Nevanlinna and Smirnov classes of analytic functions. The \textit{Nevanlinna class} $N$ consists of all analytic functions $f$ defined on $\mathbb{D}$ such that the integrals
\begin{equation*}
    \int_0^{2\pi} \log^+\left|f(re^{i\theta})\right|d\theta
\end{equation*}

\noindent are uniformly bounded for $r < 1$. For such functions the nontangential limits $f(e^{i\theta})$ exist for almost every $\theta \in [0, 2\pi]$, and $\log\left|f(e^{i\theta})\right|$ is an integrable function unless $f = 0$.

The \textit{Smirnov class} $N^+$ consists of the functions in $N$ which satisfy
\begin{equation*}
    \lim_{r \rightarrow 1} \int_0^{2\pi} \log^+\left|f(re^{i\theta})\right|d\theta = \int_0^{2\pi} \log^+\left|f(e^{i\theta})\right|d\theta
\end{equation*}

We have the following characterizations: $f \in N$ ($N^+$) if and only if $\log^+\left|f(z)\right|$ has a (quasi-bounded) harmonic majorant. Therefore, they are invariant by composition on the right with a conformal equivalence of the unit disk.

We let $P_1, P_2$ be two distinct points of $\mathbb{C}$ of norm $1$ and $\varphi : \overline{\mathbb{D}}\setminus\{P_1, P_2\} \rightarrow \overline{\mathbb{S}}$ be a surjective conformal map on the interior of $\mathbb{D}$, continuous on its domain.

$N(\mathbb{S})$ ($N^+(\mathbb{S})$) is the class of analytic functions $f$ defined on $\mathbb{S}$ such that $f\circ\varphi \in N$ $(N^+)$. These classes are closed under sum and multiplication. Notice that these definitions are independent of the conformal map $\varphi$. For more information regarding $N$ and $N^+$, see \cite{Duren2000}, \cite{rosenblum1994topics} or \cite{Shapiro01}.

We state and prove for future use the following lemma, which is probably known:

\begin{lemma}\label{lem:1.1}
Suppose that $h \in N^+$. Then the function
\begin{equation*}
    g(z) = \frac{h(z) - h(0)}{z}
\end{equation*}

with $g(0) = h'(0)$ is in $N^+$.
\end{lemma}
\begin{proof}
We have that $g$ is analytic, and belonging to $N$ is equivalent to having the integrals
\[
\int_0^{2\pi} \log (1 + \left|f(re^{i\theta})\right|)d\theta
\]

\noindent uniformly bounded in $r < 1$ \cite{Shapiro01}.

But
\begin{equation*}
    \int_0^{2\pi} \log (1 + \left|g(re^{i\theta})\right|)d\theta \leq \int_0^{2\pi} \log (1 + \left|h(re^{i\theta}) - h(0)\right|)d\theta - 2\pi \log(r)
\end{equation*}

We have that $\log(1 + \left|g(z)\right|)$ is subharmonic, and therefore the integral on the left is nondecreasing in $r$ (\cite{Duren2000}, Theorem 1.6). Since $h(z) - h(0)$ is in $N$, we have that $g \in N$.

To see that $g \in N^+$, consider for $0 < r < 1$, $0 < s$, the sets
\[
A_{r, s} = \{\theta \in [0, 2\pi] : \left|h(re^{i\theta}) - h(0)\right| > s\}
\]

Since $h(z) - h(0) \in N^+$, we have:
\begin{eqnarray*}
&& \int_0^{2\pi} \log^+ \left|\frac{h(re^{i\theta}) - h(0)}{re^{i\theta}}\right| d\theta \\
& = & \int_{A_{r, r}} \log \left|h(re^{i\theta}) - h(0)\right| d\theta - \log(r) \lambda(A_{r, r}) \\
& = & \int_0^{2\pi} \log^{+}\left|h(re^{i\theta}) - h(0)\right|d\theta + \int_{A_{r, r} \setminus A_{r, 1}} \log\left|h(re^{i\theta}) - h(0)\right| d\theta \\
&& - \log(r)\lambda(A_{r, r}) \\
& \rightarrow & \int_0^{2\pi} \log^{+}\left|h(e^{i\theta}) - h(0)\right|d\theta \\
& = & \int_0^{2\pi} \log^+ \left|g(e^{i\theta})\right| d\theta
\end{eqnarray*}
and that ends the proof.
\end{proof}

\begin{definition}\label{def:1.2}
Given an interpolation family $\{(X_z)_{z \in \partial\mathbb{S}}, X\}$ we let $\mathcal{G}$ be the space of all finite sums $g(z) = \sum\limits \psi_j(z) x_j$, where $\psi_j \in N^+(\mathbb{S})$ and $x_j \in X$, for which
\begin{equation*}
    \|g\|_{\mathcal{G}} = \esssup_{z \in \partial\mathbb{S}} \|g(z)\|_z
\end{equation*}

\noindent is finite. We denote by $\mathcal{F}$ the completion of this space with respect to $\|.\|_{\mathcal{G}}$.
\end{definition}

As in \cite{Coifman1982}, we have that $\mathcal{F}$ is a subspace of the following Banach space $\mathcal{H}$: let $X(z)$ be the closure of $X$ in $X_z$, and $K$ be the function on $\mathbb{S}$ defined by $K(z) = \exp(u(z) + iv(z))$, where $u(z) = \int_{\partial\mathbb{S}} \log k(\gamma) dP_z (\gamma)$ and $v$ is a harmonic conjugate of $u$. $K$ never vanishes and has a.\ e.\ nontangential limits on $\partial\mathbb{S}$ such that $\left|K(z)\right| = k(z)$ for a.\ e.\ $z \in \partial\mathbb{S}$. 

\begin{definition}\label{def:1.3}
The Banach space $\mathcal{H}$ consists of the $\mathcal{U}-$valued analytic functions $h$ on $\mathbb{S}$ such that $\|h(z)/K(z)\|_{\mathcal{U}}$ in bounded on $\mathbb{S}$, the nontangential limit $h(z)\in X(z) \subset \mathcal{U}$ exists for a.\ e.\ $z \in \partial\mathbb{S}$, and $\|h(z)\|_z$ is an essentially bounded measurable function on $\partial\mathbb{S}$. 

The norm of an element $h$ of $\mathcal{H}$ is
\begin{equation*}
    \|h\|_{\mathcal{H}} = \esssup_{z \in \partial\mathbb{S}} \|h(z)\|_z
\end{equation*}
\end{definition}

For $x \in X$, let
\begin{equation*}
    \|x\|_{\{z_0\}} = \inf\{\|g\|_{\mathcal{G}} : g(z_0) = x\}
\end{equation*}

\noindent and let $X_{\{z_0\}}$ be the completion of $X$ with respect to this norm.

For $z_0 \in \mathbb{S}$, we have the interpolation space $X_{[z_0]} = \{f(z_0) : f \in \mathcal{F}\}$, with the quotient norm
\begin{equation*}
    \|x\|_{[z_0]} = \inf\{\|f\|_{\mathcal{F}} : f \in \mathcal{F}, f(z_0) = x\}
\end{equation*}

The definitions of \cite{Coifman1982} are the same as presented here in the case where all vectors $x$ in the intersection $\cap_{z \in \partial\mathbb{S}} X_z$ satisfy that $\|x\|_z$ is a measurable function and $X$ is formed by all vectors in the intersection satisfying condition I1. An inspection of the proofs in that article shows that the spaces defined here satisfy interpolation properties similar to that of Calder\'{o}n's interpolation spaces. We have the following analogue of Theorem 4.1 of \cite{Coifman1982}:
\begin{theorem}\label{thm:1.4}
Let $\{(A_z)_{z \in \partial\mathbb{S}}, A\}$ and $\{(B_z)_{z \in \partial\mathbb{S}}, B\}$ be interpolation families with containing spaces $\mathcal{U}$ and $\mathcal{V}$, respectively. Let $M$ be a real function on $\partial\mathbb{S}$ such that $\log M(z)$ is absolutely integrable on $\partial\mathbb{S}$, and let $M(z_0) = \exp \int_{\partial\mathbb{S}} \log M(z) dP_{z_0}(z)$.
\begin{enumerate}
\item Let $T : A \rightarrow B$ be a linear map with $\|Ta\|_{B_z} \leq M(z)\|a\|_{A_z}$ for all $a \in A$ and for all $z \in \partial\mathbb{S}$. Then, for each $z_0 \in \mathbb{S}$, $T$ has a unique extension to an operator from $A_{\{z_0\}}$ into $B_{\{z_0\}}$ with norm at most $M(z_0)$.
\item Let $T : \mathcal{U} \rightarrow \mathcal{V}$ be a bounded operator mapping $A$ into $B$ with $\|Ta\|_{B_z} \leq M(z) \|a\|_{A_z}$ for all $a \in A$ and all $z \in \partial\mathbb{S}$. Then $T$ maps $A_{[z_0]}$ into $B_{[z_0]}$ with norm at most $M(z_0)$.
\end{enumerate}
\end{theorem}

We will always assume that $X_{\{z_0\}} = X_{[z_0]}$ isometrically, and we drop the braces and brackets and write $X_{z_0}$. The next proposition shows that this happens for every $z_0 \in \mathbb{S}$ if and only if $\mathcal{F}$ is admissible in the sense of \cite{Kalton20031131}.

\begin{props}\label{pro:1.5}
Let $\{(X_z)_{z \in \partial\mathbb{S}}, X\}$ be an interpolation family and let $\psi : \mathbb{S} \rightarrow \mathbb{D}$ be a conformal map such that $\psi(z_0) = 0$. Then $X_{\{z_0\}} = X_{[z_0]}$ if and only if for all $f \in \mathcal{F}$ with $f(z_0) = 0$ we have $\frac{f}{\psi} \in \mathcal{F}$.
\end{props}
\begin{proof}
Suppose $X_{\{z_0\}} = X_{[z_0]}$, and let $f \in \mathcal{F}$ with $f(z_0) = 0$. Then we can approximate $f$ by functions $g_n \in \mathcal{G}$ with $g_n(z_0) = 0$. Indeed, let $\epsilon > 0$, and take $g \in \mathcal{G}$ such that $\|f - g\|_{\mathcal{F}} \leq \frac{\epsilon}{3}$. So $\|g(z_0)\|_{z_0} \leq \frac{\epsilon}{3}$, and there is $h \in \mathcal{G}$ such that $h(z_0) = g(z_0)$ and $\|h\|_{\mathcal{G}} \leq \frac{2\epsilon}{3}$. Then $g - h \in \mathcal{G}$, $(g - h)(z_0) = 0$, and $\|f - (g - h)\|_{\mathcal{G}} \leq \epsilon$.

Now it is easy to check using \hyperref[lem:1.1]{Lemma 1.1} and taking $\varphi = \psi^{-1}$ that $\frac{g_n}{\psi} \in \mathcal{G}$, $\frac{f}{\psi} \in \mathcal{H}$, and that $\frac{g_n}{\psi} \rightarrow \frac{f}{\psi}$.

For the converse, it enough to use the fact that for $x \in X$ there is a function $\psi \in N^+(\mathbb{S})$ such that $\psi(z_0) = 1$ and $\psi(z) x \in \mathcal{G}$ (\cite{Coifman1982}) and prove that if $\mathcal{N} = \{f \in \mathcal{F} : f(z_0) = 0\}$, then $\mathcal{G} \cap \mathcal{N}$ is dense in $\mathcal{N}$. All one has to do is follow the same reasoning of the proof of \cite{Stafney01}, Lemma 2.5. (see also the remark (i) in Appendix 1 of \cite{Coifman1982}).
\end{proof}

This space coincides with Calder\'on's space when $X_{j + it} = X_j$, for $j = 0, 1$, $t \in \mathbb{R}$, and we take as containing space $X_0 + X_1$ and $X = X_0 \cap X_1$ (see corollary $5.1$ of \cite{Coifman1982} and Section 3, where we recall Calder\'on's method of interpolation).

We also have the following formula for the norm of $x \in X_{z_0}$ $((2.5)$ of \cite{Coifman1982}, with $p =1$): 
\begin{equation}\label{eq:1.2}
    \|x\|_{z_0} = \inf\{\exp\int_{\partial\mathbb{S}} \log \|f(z)\|_z dP_{z_0}(z)\} = \inf\{\int_{\partial\mathbb{S}}\|f(z)\|_z dP_{z_0}(z)\}
\end{equation}

\noindent where the infimum is taken over all $f \in \mathcal{F}$ with $f(z_0) = x$. Recalling that $\int_{\mathbb{S}_1} dP_{z_0}(z) = Re(z_0)$, we apply Jensen's inequality to the first equality to get
\begin{equation}\label{eq:1.3}
    \|f(z_0)\|_{z_0} \leq \Big(\int_{\mathbb{S}_0} \|f(z)\|_z dP^0_{z_0}(z)\Big)^{1 - Re(z_0)} \Big( \int_{\mathbb{S}_1} \|f(z)\|_z dP^1_{z_0}(z)\Big)^{Re(z_0)}
\end{equation}

\noindent for every $f \in \mathcal{F}$, where $dP^j_{z_0}$ are the respective probability measures defined by $dP_{z_0}$ on $\mathbb{S}_j$, $j = 0, 1$.

\subsection{Relation between interpolation and twisted sums}\label{sec:1.3}
Let $\delta_{z_0} : \mathcal{F} \rightarrow X_{z_0}$ be the evaluation at $z_0$, and $\delta'_{z_0} : \mathcal{F} \rightarrow \mathcal{U}$ be the evaluation of the derivative at $z_0$.

The following proposition is a known consequence of $\mathcal{F}$ being an admissible space of analytic functions (see \cite{Cabello2015}, for example), but we prove it here to clarify the role of \hyperref[pro:1.5]{Proposition 1.5}.

\begin{props}\label{pro:1.6}
$\delta'_{z_0} : \ker(\delta_{z_0}) \rightarrow X_{z_0}$ is bounded and onto.
\end{props}
\begin{proof}
Let $f \in \mathcal{F}$ and $f(z_0) = 0$. By \hyperref[pro:1.5]{Proposition 1.5}, if $\psi : \mathbb{S} \rightarrow \mathbb{D}$ is a conformal map such that $\psi(z_0) = 0$, we have $g = \left|\psi'(z_0)\right| \frac{f}{\psi} \in \mathcal{F}$, $g(z_0) = f'(z_0) \in X_{z_0}$, and $\|f'(z_0)\|_{z_0} \leq \|g\|_{\mathcal{F}} = \left|\psi'(z_0)\right| \|f\|_{\mathcal{F}}$. 

To see that it is onto, given $f \in \mathcal{F}$ we have $h = \psi f \in \ker(\delta_{z_0})$ and $h'(z_0) = \psi'(z_0) f(z_0)$.
\end{proof}

Therefore, we get an extension of $X_{z_0}$ by means of the pushout:
  \[
  \xymatrix{ 0 \ar[r] & \ker\delta_{z_0} \ar[r]\ar[d]^{\delta'_{z_0}} & \mathcal{F}\ar[d]
    \ar[r] & X_{z_0}\ar[r] \ar@{=}[d] &0 \\ 0 \ar[r] & X_{z_0}
    \ar[r] & PO \ar[r]& X_{z_0}\ar[r] &0}
  \]
  
If $B_{z_0}$ is a homogeneous bounded selection for $\delta_{z_0}$ and $L_{z_0}$ is a linear selection for $\delta_{z_0}$, then $PO$ is defined by the quasi-linear map $\delta'_{z_0} \omega_{z_0} : X_{z_0} \rightarrow X_{z_0}$, where $\omega_{z_0} = B_{z_0} - L_{z_0}$. Notice that $PO$ is a Banach space.

We can simplify the calculations of the next sections by working with the map $\Omega_{z_0} = \delta'_{z_0} B_{z_0}$ instead of $\omega_{z_0}$ (\cite{Castillo01}, section 3.1; \cite{Kalton20031131}, section 10): consider the space
\[
d_{\Omega_{z_0}}X_{z_0} = \{(x, y) \in \mathcal{U} \times X_{z_0} : x - \Omega_{z_0}y \in X_{z_0}\}
\]

\noindent with the quasinorm $\|(x, y)\| = \|x - \Omega_{z_0}y\|_{z_0} + \|y\|_{z_0}$. With the embedding $x \mapsto (x, 0)$ and quotient map $(x, y) \mapsto y$ this is an extension of $X_{z_0}$ equivalent to $PO$. The map $\Omega_{z_0}$ depends on the choice of $B_{z_0}$, but it is easy to see that any other choice would give an equivalent extension.

Therefore, if no confusion is to be made, we shall refer to the space $d_{\Omega_{z_0}}X_{z_0}$ simply as $dX_{z_0}$.

\section{Type and extensions induced by interpolation}\label{sec:2}

We recall the definition of (Rademacher) type of a normed space. A general reference for this topic is \cite{Albiac01}.
\begin{definition}\label{def:2.1}
Let $X$ be a normed space and $p \in [1, 2]$. $X$ has \textit{type} $p$ if there is $K>0$ such that given any finite sequence of vectors $x_1, ..., x_n \in X$, we have
\[
\mathbb{E}\Bigg\|\sum\limits_{j=1}^n \epsilon_j x_j\Bigg\| \leq K \Bigg(\sum\limits_{j=1}^n \|x_j\|^p\Bigg)^{\frac{1}{p}}
\]

\noindent where the expected value is taken over all possible choices of signs $\epsilon_j = \pm 1$.
\end{definition}

Note that every normed space has type $1$ with constant $1$. The supremum of the types of $X$ is denoted by $p_X$.

To simplify notation, we will write $\theta = Re(z_0)$.

A first relation between type and interpolation is the following classical result, here stated in our context (see \cite{Beauzamy01} and \cite{Kalton1989}, for example):
\begin{props}\label{pro:2.2}
Let $\{(X_z)_{z \in \partial\mathbb{S}}, X\}$ be an interpolation family such that $X_z$ has type $p_j$ for $z \in \mathbb{S}_j$, $j = 0, 1$, with constant $C(z)$, and that either 
\begin{enumerate}
\item $C(z)$ is an integrable function with respect to $dP_{z_0}$, or
\item $\esssup_{z \in \partial\mathbb{S}} C(z) < \infty$.
\end{enumerate}

Then $X_{z_0}$ has type $p$, where $\frac{1}{p} = \frac{1 - \theta}{p_0} + \frac{\theta}{p_1}$, with best constant at most
\begin{enumerate}
\item $C = (\int_{\mathbb{S}_0} C(z) dP^0_{z_0}(z))^{1 - \theta}(\int_{\mathbb{S}_1} C(z) dP^1_{z_0}(z))^{\theta}$
\item$C = \esssup_{z \in \partial\mathbb{S}} C(z)$. 
\end{enumerate}
respectively.
\end{props}
\begin{proof}
The proof is similar to that of Proposition 4.4 of \cite{Castillo2014}, one just has to consider general functions in $\mathcal{F}$ instead of minimal functions, and take the appropriate infimum.
\end{proof}

Our goal is to give conditions on the types of the spaces under consideration ensuring that $dX_{z_0}$ is nontrivial, or even singular. We start noticing that if we consider in $\mathcal{F}$ the norm
\[
\vertiii{F} = {\int\limits_{\partial\mathbb{S}}} \left\|F(z)\right\|_{z} dP_{z_0}(z)
\]

\noindent then by \eqref{eq:1.2} $\|x\|_{z_0} = \inf\{\vertiii{F} : F \in\mathcal{F}, F(z_0) = x\}$, for all $x \in X_{z_0}$. It is also easy to see that $\vertiii{F} \leq \|F\|_{\mathcal{F}}$, for all $F \in \mathcal{F}$, and by the same reasoning of \hyperref[pro:1.6]{Proposition 1.6} one has:
\begin{lemma}\label{lem:2.3}
$\delta'_{z_0} : (\ker\delta_{z_0}, \vertiii{.}) \rightarrow X_{z_0}$ is bounded.
\end{lemma}

Besides, since $\vertiii{F} \leq \|F\|$ on $\mathcal{F}$, we have that $B_{z_0} : X_{z_0} \rightarrow (\mathcal{F}, \vertiii{.})$ is bounded. This allows us to work with the space $(\mathcal{F}, \vertiii{.})$ instead of $(\mathcal{F}, \|.\|_{\mathcal{F}})$, which in turn will be useful in order to obtain some desired estimates.

We record for future use the following fact that follows almost by definition:
\begin{lemma}\label{lem:2.4}
If $W$ is a subspace of $X_{z_0}$ such that the twisted sum induced by $W$ with respect to $dX_{z_0}$ is trivial, then there is a linear function $\Lambda : W \rightarrow \mathcal{U}$ such that $\Omega_{z_0} |_{W} - \Lambda : W \rightarrow X_{z_0}$ is bounded.
\end{lemma}

To obtain their results on singularity, the authors of \cite{Castillo01} define estimators related to the structure of the spaces in the interpolation scheme. For example, for a K\"othe function space $X$, they define the estimator
\begin{equation*}
M_X(n) = \sup\{\|x_1 + ... + x_n\|\}
\end{equation*}
\noindent where the supremum is taken over $x_1, ..., x_n$ in the unit ball of $X$ with disjoint supports.

The following is the estimator we use:

\begin{definition}\label{def:2.5}
Let $X$ be a Banach space and $n$ a natural number. We define:
\begin{eqnarray*}
\beta_n(X) & = & \inf\Bigg\{\beta \text{ such that } \mathbb{E}\Big\|\sum\limits_{i=1}^n \epsilon_i x_i\Big\| \leq \beta \text{ , } \forall x_i \in B_{X}\Bigg\}
\end{eqnarray*}
where $B_X = \{x \in X : \|x\| \leq 1\}$.
\end{definition}

Note that $\beta_1(X) = 1$ and $\beta_{n+1}(X) \geq \beta_{n}(X)$ for all $n$.

It is easy to see that
\begin{lemma}\label{lem:2.6}
 Let $X$ be a Banach space and $x_1, ..., x_n \in X$. Then:
 \begin{eqnarray*}
  \mathbb{E}\Big\|\sum\limits_{i=1}^n \epsilon_i x_i\Big\| \leq \beta_n(X) \sup\|x_i\|
 \end{eqnarray*}
\end{lemma}

\begin{lemma}\label{lem:2.7}
Suppose that an infinite dimensional Banach space $X$ has type $p$ and that $p_X = p$. Then for every $n \geq 1$ we have $n^{\frac{1}{p}} \leq \beta_n(X) \leq K_X n^{\frac{1}{p}}$, where $K_X$ is the best constant for the type $p$ of $X$.
\end{lemma}
\begin{proof}
Fix $n$ and let $x_1, ..., x_n$ be any vectors in $B_X$, and $K$ a constant for the type $p$ of $X$. We have:
\begin{eqnarray*}
\mathbb{E}\Big\|\sum\limits_{i=1}^n \epsilon_i x_i\Big\| & \leq & K \Bigg(\sum\limits_{i=1}^n \|x_i\|^p\Bigg)^{\frac{1}{p}} \\
                    & \leq & K n^{\frac{1}{p}}
\end{eqnarray*}

So $\beta_n(X) \leq K_X n^{\frac{1}{p}}$.
On the other side, by Maurey-Pisier's theorem (see \cite{Milman01}, for example), given $\epsilon > 0$, there are vectors $x_1, ..., x_n$ of norm at most $1$ in $X$ such that, for every choice of (real) scalars $a_i$, $i = 1, ..., n$, we have
\begin{equation*}
\frac{1}{1 + \epsilon}\Bigg(\sum\limits_{i=1}^n |a_i|^p\Bigg)^{\frac{1}{p}} \leq \left\|\sum\limits_{i=1}^n a_i x_i\right\| \leq \Bigg(\sum\limits_{i=1}^n |a_i|^p\Bigg)^{\frac{1}{p}}
\end{equation*}

In particular,
\begin{equation*}
\frac{1}{1 + \epsilon} n^{\frac{1}{p}} \leq \mathbb{E}\Big\|\sum\limits_{i=1}^n \epsilon_i x_i\Big\|
\end{equation*}

Since $\epsilon$ was arbitrary, $n^{\frac{1}{p}} \leq \beta_n(X)$.
\end{proof}

\begin{lemma}\label{lem:2.8} Suppose that $\Omega : X \rightarrow X$ is a quasi-linear map that is trivial on a closed subspace $W \subset X$. Then there is $K_1 > 0$ such that for every sequence $(w_i)_{i=1}^n$ in $B_W$,
\begin{equation*}
\mathbb{E}\Bigg\|\Omega\Bigg(\sum\limits_{i=1}^n \epsilon_i w_i\Bigg) - \sum\limits_{i=1}^n \epsilon_i\Omega(w_i)\Bigg\| \leq K_1 \beta_n(X)
\end{equation*}
\end{lemma}

\begin{proof}
Since $\Omega{\restriction_{W}}$ is trivial, there is a linear map $\Lambda : W \rightarrow X$ with $\Omega{\restriction_{W}} - \Lambda$ bounded. Then:
\begin{eqnarray*}
&\Bigg\|\Omega\Bigg(\sum\limits_{i=1}^n \epsilon_i w_i\Bigg) - \sum\limits_{i=1}^n \epsilon_i\Omega(w_i)\Bigg\|  \leq& \\  &\Bigg\|\Omega\Bigg(\sum\limits_{i=1}^n \epsilon_i w_i\Bigg) - \Lambda\Bigg(\sum\limits_{i=1}^n \epsilon_i w_i\Bigg)\Bigg\| + \Bigg\|\sum\limits_{i=1}^n \epsilon_i(\Lambda - \Omega)(w_i)\Bigg\|&
\end{eqnarray*}
\vspace{0.2cm}

So:
\begin{eqnarray*}
\mathbb{E}\Bigg\|\Omega\Bigg(\sum\limits_{i=1}^n \epsilon_i w_i\Bigg) - \sum\limits_{i=1}^n \epsilon_i\Omega(w_i)\Bigg\| & \leq & \|\Omega - \Lambda\| \mathbb{E}\Big\|\sum\limits_{i=1}^n \epsilon_i w_i\Big\| + \mathbb{E}\Big\|\sum\limits_{i=1}^n \epsilon_i (\Lambda - \Omega)(w_i)\Big\| \\
      & \leq & \|\Omega - \Lambda\|\beta_n(X) + \beta_n(X) \sup\|(\Omega - \Lambda)(w_i)\| \\
      & = & 2\|\Omega - \Lambda\| \beta_n(X)
\end{eqnarray*}
\end{proof}

\begin{rmk}
Note that in the setting we work, that is, of twisted sums induced by interpolation, $\Lambda$ does not have its image in $X_{z_0}$, but $\Omega_{z_0} - \Lambda$ does (\hyperref[lem:2.4]{Lemma 2.4}), and thus the result is still valid in this context.
\end{rmk}

\begin{definition}\label{def:2.9}
We define $\beta_{n, j} = \esssup_{z \in \mathbb{S}_j} \beta_n(X_z)$, whenever it makes sense.
\end{definition}

\begin{lemma}\label{lem:2.10}
Let $\{(X_z)_{z \in \partial\mathbb{S}}, X\}$ be an interpolation family for which 

\noindent $\esssup_{z \in \mathbb{S}_j} \beta_n(X_z)$ is finite, $j = 0, 1$. Then $\beta_n(X_{z_0}) \leq \beta_{n, 0}^{1 - \theta} \beta_{n, 1}^{\theta}$.
\end{lemma}
\begin{proof}
 Let $x_1, ..., x_n \in B_{X_{z_0}}$, and take $F_j \in \mathcal{F}$ such that $F_j(z_0) = x_j$, $j = 1, ..., n$.
 
Then, for every choice of signs $\epsilon_j$, by \eqref{eq:1.3} we have:
\begin{eqnarray*}
&\left\|\sum\limits_{j=1}^n \epsilon_j x_j\right\|_{z_0} & \\
&\leq \Bigg( \bigints\limits_{\mathbb{S}_0} \left\|\sum\limits_{j=1}^n \epsilon_j F_j(z)\right\|_{z} dP^0_{z_0} (z) \Bigg)^{1 - \theta} \Bigg( {\bigints\limits_{\mathbb{S}_1}} \left\|\sum\limits_{j=1}^n \epsilon_j F_j(z)\right\|_{z} dP^1_{z_0} (z) \Bigg)^{\theta}&
\end{eqnarray*}

Then
\begin{eqnarray*}
&\sum_{\substack{+ \\ -}}\left\|\sum\limits_{j=1}^n \epsilon_j x_j\right\|_{z_0} & \\
& \leq \sum_{\substack{+ \\ -}} \Bigg( {\bigints\limits_{\mathbb{S}_0}} \left\|\sum\limits_{j=1}^n \epsilon_j F_j(z)\right\|_{z} dP^0_{z_0} (z) \Bigg)^{1 - \theta} \Bigg( {\bigints\limits_{\mathbb{S}_1}} \left\|\sum\limits_{j=1}^n \epsilon_j F_j(z)\right\|_{z} dP^1_{z_0} (z) \Bigg)^{\theta}&
\end{eqnarray*}

Using H\"{o}lder's inequality with $\frac{1}{\frac{1}{1-\theta}} + \frac{1}{\frac{1}{\theta}} = 1$, we have:
\begin{eqnarray*}
&\sum_{\substack{+ \\ -}}\left\|\sum\limits_{j=1}^n \epsilon_j x_j\right\|_{z_0}   & \\
\leq &\Bigg(\sum_{\substack{+ \\ -}} {\bigints\limits_{\mathbb{S}_0}} \left\|\sum\limits_{j=1}^n \epsilon_j F_j(z)\right\|_{z} dP^0_{z_0} (z) \Bigg)^{1 - \theta} \Bigg(\sum_{\substack{+ \\ -}} {\bigints\limits_{\mathbb{S}_1}} \left\|\sum\limits_{j=1}^n \epsilon_j F_j(z)\right\|_{z} dP^1_{z_0} (z) \Bigg)^{\theta} &\\
		= & \Bigg( {\bigints\limits_{\mathbb{S}_0}} \sum_{\substack{+ \\ -}}\left\|\sum\limits_{j=1}^n \epsilon_j F_j(z)\right\|_{z} dP^0_{z_0} (z) \Bigg)^{1 - \theta} \Bigg( {\bigints\limits_{\mathbb{S}_1}}\sum_{\substack{+ \\ -}} \left\|\sum\limits_{j=1}^n \epsilon_j F_j(z)\right\|_{z} dP^1_{z_0} (z) \Bigg)^{\theta}&
\end{eqnarray*}

Then: 
\begin{eqnarray*}
& \mathbb{E}\left\|\sum\limits_{j=1}^n \epsilon_j x_j\right\|_{z_0}  &\\
 &\leq \Bigg( {\bigints\limits_{\mathbb{S}_0}} \mathbb{E}\left\|\sum\limits_{j=1}^n \epsilon_j F_j(z)\right\|_{z} dP_{0, \theta} (z) \Bigg)^{1 - \theta} \Bigg( {\bigints\limits_{\mathbb{S}_1}}\mathbb{E} \left\|\sum\limits_{j=1}^n \epsilon_j F_j(z)\right\|_{z} dP^1_{z_0} (z) \Bigg)^{\theta} &  \\
								       &\leq \Bigg( {\bigints\limits_{\mathbb{S}_0}} \beta_n(X_z) \sup\limits_j \|F_j(z)\|_{z} dP^0_{z_0} (z) \Bigg)^{1 - \theta} \Bigg( {\bigints\limits_{\mathbb{S}_1}}\beta_n(X_z) \sup\limits_j \|F_j(z)\|_{z} dP^1_{z_0} (z) \Bigg)^{\theta} & \\
								        & \leq \beta_{n, 0}^{1 - \theta} \beta_{n, 1}^{\theta} \sup\|F_j\|_{\mathcal{F}}&
\end{eqnarray*}
Since the $F_j$ were arbitrary, we can take $\sup\|F_j\|_{\mathcal{F}}$ as close to $\sup\|x_j\|_{z_0}$ as we wish, and we have the result.
\end{proof}

\begin{lemma}\label{lem:2.11}
 Let $\{(X_z)_{z \in \partial\mathbb{S}}, X\}$ be an interpolation family such that the spaces $X_z$, for a.\ e.\ $z \in \mathbb{S}_j$, have type $p_j$ with uniformly bounded constants and that $p_{X_z} = p_j$, $j = 0, 1$. There is $K_2 > 0$ such that, if $x_1, ..., x_n \in B_{X_{z_0}}$
 \begin{eqnarray*}
  \mathbb{E}\Bigg\|\Bigg(\sum\limits_{j=1}^n \epsilon_i x_j\Bigg)\log\frac{\beta_{n, 0}}{\beta_{n, 1}} - \Omega_{z_0} \Bigg(\sum\limits_{j=1}^n \epsilon_j x_j\Bigg) + \sum\limits_{j=1}^n \Omega_{z_0}(\epsilon_j x_j)\Bigg\| \leq K_2 \beta_{n, 0}^{1-\theta}\beta_{n, 1}^\theta
 \end{eqnarray*}
\end{lemma}
\begin{proof}
The proof is similar to that of Lemma 4.8 of \cite{Castillo01}.
Define $\beta_n(z) = \beta_{n, 0}^{1-z}\beta_{n, 1}^z$.
Let $x_1, ..., x_n$ be vectors in $B_{X_{z_0}}$, and define for each $\epsilon_J = \{\epsilon_1, ..., \epsilon_n\} \in \{-1, 1\}^n$ the following function:

\begin{equation*}
F_{\epsilon_J}(z) = \frac{\sum\limits_{j=1}^n \epsilon_j B_{z_0}(x_j)(z)}{\beta_n(z)}
\end{equation*}
\vspace{0.2cm}

Then $F_{\epsilon_J} \in \mathcal{F}$ because of \hyperref[lem:2.7]{Lemma 2.7}.

We have
\begin{equation}\label{eq:2.1}
\mathbb{E}\vertiii{F_{\epsilon_J}} \leq  \|B_{z_0}\|
\end{equation}

\noindent where the expected value is taken over all choices of signs $\epsilon_J \in \{-1, 1\}^n$. Indeed:
\begin{eqnarray*}
\mathbb{E}\vertiii{F_{\epsilon_J}} & = & \mathbb{E}{\int\limits_{\mathbb{S}_0}} \left\|F_{\epsilon_J}(z)\right\|_{z} dP_{z_0}(z) + \mathbb{E}{\int\limits_{\mathbb{S}_1}} \left\|F_{\epsilon_J}(z)\right\|_{z} dP_{z_0}(z) \\
        & = & \frac{1}{\beta_{n, 0}}{\bigints\limits_{\mathbb{S}_0}} \mathbb{E}\left\|\sum\limits_{j=1}^n \epsilon_j B_{z_0}(x_j)(z)\right\|_z dP_{z_0}(z) + \\ 
        && + \frac{1}{\beta_{n, 1}}{\bigints\limits_{\mathbb{S}_1}} \mathbb{E}\left\|\sum\limits_{j=1}^n \epsilon_j B_{z_0}(x_j)(z)\right\|_z dP_{z_0}(z) \\       
        & \leq & \|B_{z_0}\|
\end{eqnarray*}

Also:
\begin{eqnarray*}
\beta_n(z_0)F_{\epsilon_J}'(z_0)
        & = & \Bigg(\sum\limits_{j=1}^n \epsilon_i x_j\Bigg)\log\frac{\beta_{n, 0}}{\beta_{n, 1}} + \sum\limits_{j=1}^n \epsilon_j \Omega_{z_0}(x_j)
\end{eqnarray*}

Thus:
\begin{eqnarray*}
& \delta'_{z_0} (\beta_n(z_0)F_{\epsilon_J}) -\Omega_{z_0} \Bigg(\sum\limits_{i=1}^n \epsilon_i x_i\Bigg) = & \\
&\Bigg(\sum\limits_{j=1}^n \epsilon_i x_j\Bigg)\log\frac{\beta_{n, 0}}{\beta_{n, 1}} - \Omega_{z_0} \Bigg(\sum\limits_{j=1}^n \epsilon_j x_j\Bigg) + \sum\limits_{j=1}^n \Omega_{z_0}(\epsilon_j x_j)&
\end{eqnarray*}

Also, $\beta_n(z_0)F_{\epsilon_J} - B_{z_0}\Bigg(\sum\limits_{j=1}^n \epsilon_j x_j\Bigg)$ is an element of $\ker\delta_{z_0}$.
So, by \hyperref[lem:2.3]{Lemma 2.3}:
\begin{eqnarray*}
 && \Bigg\|\Bigg(\sum\limits_{j=1}^n \epsilon_i x_j\Bigg)\log\frac{\beta_{n, 0}}{\beta_{n, 1}} - \Omega_{z_0} \Bigg(\sum\limits_{j=1}^n \epsilon_j x_j\Bigg) + \sum\limits_{j=1}^n \Omega_{z_0}(\epsilon_j x_j)\Bigg\| \\
 & = & \Bigg\|\delta'_{z_0} (\beta_n(z_0)F_{\epsilon_J}) -\Omega_{z_0} \Bigg(\sum\limits_{i=1}^n \epsilon_i x_i\Bigg)\Bigg\| \\
                & \leq &  \|\delta'_{z_0}\|\vertiii{\beta_n(z_0)F_{\epsilon_J} - B_{z_0}\Bigg(\sum\limits_{j=1}^n \epsilon_j x_j\Bigg)} \\
                & \leq & \|\delta'_{z_0}\|\vertiii{\beta_n(z_0)F_{\epsilon_J}} + \|\delta'_{z_0}\|\vertiii{B_{z_0}\Bigg(\sum\limits_{j=1}^n \epsilon_j x_j\Bigg)} \\
                & \leq & \left|\beta_n(z_0)\right| \|\delta'_{z_0}\|\vertiii{F_{\epsilon_{J}}} + \|\delta'_{z_0}\|\|B_{z_0}\|\Bigg\|\sum\limits_{j=1}^n \epsilon_j x_j\Bigg\|
\end{eqnarray*}

By \eqref{eq:2.1} and \hyperref[lem:2.10]{Lemma 2.10}, taking the expected value:
\begin{eqnarray*}
 && \mathbb{E} \Bigg\|\Bigg(\sum\limits_{j=1}^n \epsilon_i x_j\Bigg)\log\frac{\beta_{n, 0}}{\beta_{n, 1}} - \Omega_{z_0} \Bigg(\sum\limits_{j=1}^n \epsilon_j x_j\Bigg) + \sum\limits_{j=1}^n \Omega_{z_0}(\epsilon_j x_j)\Bigg\| \\
 & \leq & \left|\beta_n(z_0)\right| \|\delta'_{z_0}\|\|B_{z_0}\| + \|\delta'_{z_0}\|\|B_{z_0}\| \beta_n(X_{z_0}) \\
	    & \leq & 2 \|\delta'_{z_0}\|\|B_{z_0}\|\left|\beta_n(z_0)\right|
\end{eqnarray*}

Take then $K_2 = 2 \|\delta'_{z_0}\|\|B_{z_0}\|$.
\end{proof}

\begin{lemma}\label{lem:2.12}
Let $\{(X_z)_{z \in \partial\mathbb{S}}, X\}$ be an interpolation family for which the spaces $X_z$, for a.\ e.\ $z \in \mathbb{S}_j$, have type $p_j$ with uniformly bounded constants and satisfy $p_{X_z} = p_j$, $j = 0, 1$ and let $z_0 \in \mathbb{S}$. Suppose that $\Omega_{z_0}{\restriction_{W}}$ is trivial, where $W$ is an infinite dimensional closed subspace of $X_{z_0}$. Then there is $K > 0$ such that for every natural $n$
\begin{eqnarray*}
\Bigg|\log\frac{\beta_{n, 0}}{\beta_{n, 1}}\Bigg| \beta_n(W) \leq K \beta_{n, 0}^{1 - \theta} \beta_{n, 1}^{\theta}
\end{eqnarray*}
\end{lemma}
\begin{proof}
Since $\Bigg|\log\frac{\beta_{n, 0}}{\beta_{n, 1}}\Bigg| \Bigg\|\sum\limits_{j=1}^n \epsilon_j x_j\Bigg\|$ is bounded above by
\begin{eqnarray*}
 & \Bigg\|\Bigg(\sum\limits_{j=1}^n \epsilon_i x_j\Bigg)\log\frac{\beta_{n, 0}}{\beta_{n, 1}} - \Omega_{\theta} \Bigg(\sum\limits_{j=1}^n \epsilon_j x_j\Bigg) + \sum\limits_{j=1}^n \Omega_{z_0}(\epsilon_j x_j)\Bigg\| +   \\
      & + \Bigg\|\Omega_{z_0}\Bigg(\sum\limits_{j=1}^n \epsilon_j x_j\Bigg) - \sum\limits_{j=1}^n \epsilon_j\Omega_{z_0}(x_j)\Bigg\|&
\end{eqnarray*}

\noindent we have the same inequality for the expected value. Thus, by lemmas \ref{lem:2.8}, \ref{lem:2.10} and \ref{lem:2.11},
\begin{eqnarray*}
 \Bigg|\log\frac{\beta_{n, 0}}{\beta_{n, 1}}\Bigg| \mathbb{E}\Bigg\|\sum\limits_{j=1}^n \epsilon_j x_j\Bigg\| \leq K \beta_{n, 0}^{1 - \theta}\beta_{n, 1}^{\theta}
\end{eqnarray*}
\noindent whenever $x_1, ..., x_n \in B_{W}$.
\end{proof}

\begin{theorem}\label{thm:2.13}
 Let $\{(X_z)_{z \in \partial\mathbb{S}}, X\}$ be an interpolation family for which the spaces $X_z$, for a.\ e.\ $z \in \mathbb{S}_j$, have type $p_j$ with uniformly bounded constants, satisfy $p_{X_z} = p_j$, $j=0, 1$, and $p_0 \neq p_1$. Consider $p$ given by $\frac{1}{p} = \frac{1 - Re(z_0)}{p_0} + \frac{Re(z_0)}{p_{1}}$. Suppose also that $W$ is an infinite dimensional closed subspace of $X_{z_0}$.
 
 \textit{a)} If $p_W = p$, then $\Omega_{z_0}{\restriction_{W}}$ is not trivial.
 
 \textit{b)} If $X_{z_0}$ is infinite dimensional and $p_W = p$ for every $W$, then $\Omega_{z_0}$ is singular.
\end{theorem}
\begin{proof}

\textit{a)} By the previous lemma, given $W$ such that $\Omega_{z_0}{\restriction_{W}}$ is trivial, we have $K>0$ such that
\begin{eqnarray*}
\Bigg|\log\frac{\beta_{n, 0}}{\beta_{n, 1}}\Bigg| \beta_n(W) \leq K \beta_{n, 0}^{1 - Re(z_0)} \beta_{n, 1}^{Re(z_0)}
\end{eqnarray*}

But by \hyperref[lem:2.7]{Lemma 2.7}, $n^{\frac{1}{p}} \leq \beta_n(W)$ and $\beta_{n, 0}^{1 - Re(z_0)} \beta_{n, 1}^{Re(z_0)} \leq C n^{\frac{1}{p}}$, for some $C >0$, and therefore $\Bigg|\log\frac{\beta_{n, 0}}{\beta_{n, 1}}\Bigg|$ is bounded. However, again by \hyperref[lem:2.7]{Lemma 2.7}, there are constants $C_1, C_2 >0$ for which
\begin{eqnarray*}
C_1 n^{\frac{1}{p_{0}} - \frac{1}{p_{1}}} \leq \frac{\beta_{n, 0}}{\beta_{n, 1}} \leq C_2 n^{\frac{1}{p_{0}} - \frac{1}{p_{1}}}
\end{eqnarray*}

If $p_{0} > p_{1}$, then $\frac{\beta_{n, 0}}{\beta_{n, 1}} \rightarrow 0$. If $p_{0} < p_{1}$, then $ \frac{\beta_{n, 0}}{\beta_{n, 1}} \rightarrow \infty$. In either case we have a contradiction.
Item \textit{b)} is a direct consequence of \textit{a)}.
\end{proof}

We can restate the theorem as:
\begin{theorem}\label{thm:2.14} [2.13 restated]
 Let $\{(X_z)_{z \in \partial\mathbb{S}}, X\}$ be an interpolation family for which the spaces $X_z$, for a.\ e.\ $z \in \mathbb{S}_j$ have type $p_j$ with uniformly bounded constants, satisfy $p_{X_z} = p_j$, $j=0, 1$, and $p_0 \neq p_1$. Consider $p$ given by $\frac{1}{p} = \frac{1 - Re(z_0)}{p_0} + \frac{Re(z_0)}{p_{1}}$.
 
 \textit{a)} Suppose that $W \subset X_{z_0}$ is an infinite dimensional closed subspace such that $p_W = p$. Then the twisted sum induced by $W$ is not trivial. In particular, if $X_{z_0}$ is infinite dimensional and $p_{X_{z_0}} = p$, $dX_{z_0}$ is a nontrivial extension of $X_{z_0}$.
 
 \textit{b)} If $X_{z_0}$ is infinite dimensional and $p_W = p$ for every infinite dimensional closed subspace $W \subset X_{z_0}$, then $dX_{z_0}$ is a singular extension of $X_{z_0}$.
\end{theorem}

\section{Cotype and extensions induced by interpolation}\label{sec:3}

We recall now the definition of (Rademacher) cotype of a normed space. Once again, we refer to \cite{Albiac01} for more information.
\begin{definition}\label{def:3.1}
Let $X$ be a normed space and $q \in [2, \infty]$. $X$ has \textit{cotype} $q$ if there is $K > 0$ such that, given any finite sequence of vectors $x_1, ..., x_n \in X$, we have
\[
\Bigg(\sum\limits_{j=1}^n \|x_j\|^q\Bigg)^{\frac{1}{q}} \leq K \mathbb{E}\Bigg\|\sum\limits_{j=1}^n \epsilon_j x_j\Bigg\|
\]

\noindent where the expected value is taken over all possible choices of signs $\epsilon_j = \pm 1$.
\end{definition}

Note that every space has cotype $\infty$ with constant $1$.

We have the following result due to Pisier \cite{Pisier02}:
\begin{theorem}\label{thm:3.2} Let $X$ be a Banach space with type strictly bigger than $1$. Then $X$ has cotype $q$ if and only if $X^*$ has type $p$, with $\frac{1}{p} + \frac{1}{q} = 1$.
\end{theorem}

A Banach space with nontrivial type is said to be \textit{B-convex}. This result suggests we might prove the analogue of \hyperref[thm:2.14]{Theorem 2.14} for cotype by an argument of dualization.

First of all, is the dual of an interpolation space the interpolation of the duals? This is not true in general (see \cite{Coifman1982}), so we simplify things by considering complex interpolation for a couple of Banach spaces, which we recall now.

\subsection{Interpolation of couples of Banach spaces}\label{sec:3.1}
We recall now the classical method of complex interpolation for a couple of Banach spaces due to Calder\'{o}n and independently to Lions. For a detailed exposition, see \cite{Bergh01}. Let $\overline{X} = (X_0, X_1)$ be a couple of Banach spaces. The couple is \textit{compatible} if the spaces are continuously embedded in a Hausdorff topological vector space, which we may then replace by the space
\begin{equation*}
    \Sigma(\overline{X}) = \{x_0 + x_1 : x_0 \in X_0, x_1 \in X_1\}
\end{equation*}

\noindent with the complete norm
\begin{equation*}
    \|x\|_{\Sigma(\overline{X})} = \inf\{\|x_0\|_{X_0} + \|x_1\|_{X_1} : x = x_0 + x_1, x_0 \in X_0, x_1 \in X_1\}
\end{equation*}

We also have the intersection space $\Delta(\overline{X}) = X_0 \cap X_1$ with the complete norm
\begin{equation*}
    \|x\|_{\Delta(\overline{X})} = \max\{\|x\|_{X_0}, \|x\|_{X_1}\}
\end{equation*}

Let $\mathcal{A}$ be the space of functions $f$ on $\overline{\mathbb{S}}$ with values in $\Sigma(\overline{X})$ such that:
\begin{itemize}
\item $f$ is $\Sigma(\overline{X})$-bounded and continuous on $\overline{\mathbb{S}}$;
\item $f$ is $\Sigma(\overline{X})$-analytic on $\mathbb{S}$;
\item The functions $t \mapsto f(j + it)$ are continuous functions from $\mathbb{R}$ into $X_j$, which tend to zero as $\left|t\right| \rightarrow \infty$, $j = 0, 1$.
\end{itemize}

This last condition of $t \mapsto f(j+it)$ tending to zero may be replaced by them being bounded.

If we consider on $\mathcal{A}$ the norm
\begin{equation*}
    \|f\|_{\mathcal{A}} = \max\{\sup\|f(it)\|_{X_0}, \sup\|f(1 + it)\|_{X_1}\}
\end{equation*}

\noindent then $\mathcal{A}$ is a Banach space for which the evaluation $\delta_{z_0}$ at $z_0 \in \mathbb{S}$ is a continuous map, and one has the interpolation space $X_{z_0} = \mathcal{A}/\ker(\delta_{z_0})$. It is easy to see that $X_{z_0} = X_{Re(z_0)}$ isometrically, so that most of the times we can restrict ourselves to work with $X_{\theta}$, $\theta \in (0, 1)$.

As happens for the interpolation of families of Banach spaces, the interpolation of couples induces an extension of the interpolation space. We can view the couple $(X_0, X_1)$ as a family $\{(X_z)_{z \in \partial\mathbb{S}}, X\}$, where $X_{j + it} = X_j$, $j = 0, 1$, the containing space is $\Sigma(\overline{X})$ and $X = \Delta(\overline{X})$, and the two methods give us the same interpolation spaces (Corollary 5.1 of \cite{Coifman1982}). The space $\mathcal{A}$ is a closed subspace of $\mathcal{F}$ (a consequence of Lemma 4.2.3 of \cite{Bergh01}), and since any two choices of a homogeneous bounded selection for $\delta_{\theta}$ give rise to the same extension of $X_{\theta}$, we have that the extensions defined by the two interpolation methods are actually the same. 

The following classical results (\cite{Bergh01}, 4.5.1 and 4.5.2) yield conditions that guarantee that the interpolation of the duals is the dual of the interpolation. The authors of \cite{Bergh01} reference Lions and Peetre \cite{Lions1964} and Calder\'on \cite{Calderon1964}.

\begin{theorem}\label{thm:3.3}
Suppose that $\Delta(\overline{X})$ is dense in $X_0$ and in $X_1$ and that at least one of the spaces $X_0$ or $X_1$ is reflexive. Then $(X_0, X_1)_{\theta}^* = (X_0^*, X_1^*)_{\theta}$ isometrically.
\end{theorem}

Therefore, we may write $(X_{\theta})^* = (X_0^*, X_1^*)_{\theta} = X_{\theta}^*$.

If we then dualize the twisted sum
  \[
  \xymatrix{ 0 \ar[r] & X_{\theta} \ar[r] & d_{\Omega_{\theta}}X_{\theta}
    \ar[r] & X_{\theta}\ar[r]  &0}
  \]

\noindent we obtain an extension of $X_{\theta}^*$
\begin{equation}\label{eq:3.1}
  \xymatrix{ 0 \ar[r] & X_{\theta}^* \ar[r] & (d_{\Omega_{\theta}}X_{\theta})^*
    \ar[r] & X_{\theta}^* \ar[r]  &0}
\end{equation}

Since we are dealing with the interpolation of the couple $\overline{X^*}$, we have the evaluation at $\theta$ from $\mathcal{F}(\overline{X^*})$ into $X^*_{\theta}$, which we shall denote by $\delta_{\theta}$ too. Similarly, we have the evaluation of the derivative at $\theta$ with respect to the interpolation of the couple $\overline{X^*}$, still denoted by $\delta_{\theta}'$.

If we denote by $B_{\theta}^*$ a homogeneous bounded selection for the evaluation at $\theta$ of the interpolation of the couple $\overline{X^*}$, we have an extension of $X_{\theta}^*$ induced by interpolation
\begin{equation}\label{eq:3.2}
  \xymatrix{ 0 \ar[r] & X_{\theta}^* \ar[r] & d_{\Omega_{\theta}^*}X_{\theta}^*
    \ar[r] & X_{\theta}^* \ar[r]  &0}
\end{equation}
  
\noindent where $\Omega_{\theta}^* = \delta'_{\theta} B_{\theta}^*$. So, if the twisted sums \eqref{eq:3.1} and \eqref{eq:3.2} are in some way equivalent, we can apply \hyperref[thm:2.14]{Theorem 2.14} to the twisted sum \eqref{eq:3.2} to obtain a similar result involving cotype.

In the language of the theory of duality of twisted sums (\cite{Sanchez02}, \cite{Castillo2007}) we look for the quasi-linear map conjugate to $\omega_{\theta}$ (recall from the introduction that $\omega_{\theta} = \Omega_{\theta} - L_{\theta}$).

The twisted sums \eqref{eq:3.1} and \eqref{eq:3.2} are (isomorphically) equivalent (see the survey \cite{Cwikel01} and \cite{Rochberg01}). In this way, we obtain a result analogue to the triviality part of \hyperref[thm:2.14]{Theorem 2.14} for cotype.

Denoting by $q_X = \inf\{q : X \text{ has cotype } q\}$, by \hyperref[thm:3.2]{Theorem 3.2}, if $X$ has type bigger than $1$, then $\frac{1}{p_{X^*}} + \frac{1}{q_X} = 1$. By noting that the dual of a trivial twisted sum is trivial, we can use \hyperref[thm:2.14]{Theorem 2.14} to obtain:
\begin{theorem}\label{thm:3.4}
Let $\overline{X} = (X_0, X_1)$ be a compatible pair of Banach spaces. Suppose $\Delta(\overline{X})$ dense in $X_0$ and in $X_1$ and that at least one of the spaces $X_0$ or $X_1$ is reflexive. Suppose also that $X_0$ and $X_1$ have type strictly bigger than $1$, that $X_0$ has cotype $q_{X_0}$, that $X_1$ has cotype $q_{X_1}$, $q_{X_0} \neq q_{X_1}$, that $X_{\theta}$ is infinite dimensional and that $q_{X_{\theta}}$ satisfies $\frac{1}{q_{X_{\theta}}} = \frac{1-\theta}{q_{X_0}} + \frac{\theta}{q_{X_1}}$. Then $d_{\Omega_{\theta}}X_{\theta}$ is a nontrivial extension of $X_{\theta}$.
\end{theorem}

\section{Examples}\label{sec:4}

In this section we mention some already known examples for illustrative purposes, showing in particular that \hyperref[thm:2.14]{Theorem 2.14} is optimal. In \hyperref[sec:4.3]{Section 4.3}, based on \cite{Fuente2014}, we study submodules of the Schatten classes, and in \hyperref[sec:4.4]{Section 4.4} we show new nontrivial twisted sums in which the three spaces in the exact sequence do not have the Compact Approximation Property.

\subsection{\texorpdfstring{$\ell_p$}{lp} spaces}\label{sec:4.1}

Consider the interpolation scale $(\ell_{\infty}, \ell_1)_{\theta} = \ell_p$, where $p = \frac{1}{\theta}$. We have the extensions defined by interpolation:
  \begin{equation}\label{eq:4.1}
  \xymatrix{ 0 \ar[r] & \ell_p \ar[r] & Z_p
    \ar[r] & \ell_p\ar[r]  &0}
  \end{equation}
  
It is already known that these extensions are singular and cosingular, for every $1 < p < \infty$. Using theorems \ref{thm:2.14} and \ref{thm:3.4} and reiteration, we obtain singularity of $(4.1)$ for $1 < p < 2$, nontriviality for $2 < p < \infty$, and by duality cosingularity for $2 < p < \infty$.

\subsection{\texorpdfstring{$L_p$}{Lp} spaces}\label{sec:4.2}
Consider the $L_p$ spaces over $[0, 1]$. We have the interpolation scale $(L_{\infty}[0,1], L_1[0,1])_{\theta} = L_p[0,1]$, where $p = \frac{1}{\theta}$. We then get extensions
  \begin{equation}\label{eq:4.2}
  \xymatrix{ 0 \ar[r] & L_p[0,1] \ar[r] & LZ_p
    \ar[r] & L_p[0,1]\ar[r]  &0}
  \end{equation}

For $1 < p < \infty$ these are nontrivial twisted sums, but unlike the previous case these are \textit{not} singular (\cite{Fuente2013}, see also \cite{Sanchez01}). Actually, the extension is trivial on the copy of $\ell_2$ spanned by the Rademacher functions, and if $1 < p < q < 2$, then it is trivial on a copy of $\ell_q$. This shows that \hyperref[thm:2.14]{Theorem 2.14} is optimal. Again, we can use theorems \ref{thm:2.14} and \ref{thm:3.4} to get nontriviality for $p \neq 2$.

\subsection{Schatten classes}\label{sec:4.3}

This section treats of a particular case of interpolation between a von Neumann algebra and its predual. For a study of twisted sums in this context, see \cite{CabelloSanchez2016}.

Let $H$ be the separable Hilbert space. For $1 \leq p < \infty$, the Schatten class $C_p$ (we follow the notation of \cite{Arazy1975}) is the space of compact operators $x$ on $H$ such that $\|x\|_p = (trace(x^*x)^{\frac{p}{2}})^{\frac{1}{p}} < \infty$. With the norm $\|.\|_p$, $C_p$ is a Banach space.

The space $C_p$ can be viewed as a noncommutative version of the function space $L_p$. It reproduces much of the behaviour of its commutative version. For example, $C_p$ has type $\max\{p, 2\}$. Since we can embed $l_p$ in $C_p$ as diagonal operators, we have that $p_{C_p} = \max\{p, 2\}$.

Also, $B(H)$ is the noncommutative analog of $L_{\infty}$, and by embedding $C_1$ in $B(H)$, we see that $(B(H), C_1)$ is a compatible couple, and actually $(B(H), C_1)_{\theta} = C_p$, with $\theta = \frac{1}{p}$ \cite{Pisier2003}.

That means that for $1 < p < \infty$ we can twist $C_p$ with itself to get an extension $\Theta_p$ of $C_p$ induced by interpolation. In \cite{Fuente2014} the singularity of this extension is studied, and the following result is obtained:
\begin{theorem}\label{thm:4.1}
Let $V$ be a closed (left, right or bilateral) $B(H)$-submodule of $C_p$ for $1 < p < \infty$. The following are equivalent:
\begin{itemize}
    \item[(a)] The twisted sum induced by $V$ is nontrivial.
    \item[(b)] $\max\{rk(T) : T \in V\} = \infty$, where $rk(T)$ is the rank of $T$.
\end{itemize}
\end{theorem}

Using this theorem and our result on nontriviality and type, we get the following result, which to the best of our knowledge is new.
\begin{theorem}\label{thm:4.2}
Let $1 < p < 2$, and let $V$ be a closed infinite dimensional (left, right or bilateral) $B(H)$-submodule of $C_p$ such that $\max\{rk(T) : T \in V\} < \infty$. Then, viewed as a subspace of $C_p$:
\begin{itemize}
    \item[(a)] $V$ does not contain an isomorphic copy of $C_p$.
    \item[(b)] $V$ does not contain an isomorphic copy of $\ell_p$.
    \item[(c)] Every basic sequence in $V$ admits a subsequence equivalent to the canonical basis of $\ell_2$ such that its span is complemented in $C_p$.
\end{itemize}
\end{theorem}
\begin{proof}
First we notice that by using reiteration we can considerate the interpolation as happening between $C_2$ and $C_1$. In this case, $C_p = (C_2, C_1)_{\eta}$, where $\frac{1}{p} = \frac{1 - \eta}{2} + \eta$, and the conditions on $p_{C_1}, p_{C_2}$ and $p_{C_p}$ of \hyperref[thm:2.14]{Theorem 2.14} are satisfied.

Now, if $V$ contained a copy of either $C_p$ or $\ell_p$ then it would follow that $p_V = p$, and the twisted sum induced by $V$ would not be trivial. This proves $(a)$ and $(b)$.

Assertion $(c)$ follows from $(b)$ and \cite{Arazy1975}, Theorem 1.
\end{proof}

Notice that if $V$ is bilateral, then it trivially satisfies $\max\{rk(T) : T \in V\} = \infty$. Notice also that the Schatten classes have a natural basis, but we do not know if \hyperref[thm:4.2]{Theorem 4.2} can be obtained from the results of \cite{Castillo01}.

\subsection{Spaces without the CAP}\label{sec:4.4}
In \cite{Szankowski1978}, Szankowski gave examples of subspaces of $\ell_p$ without the Compact Approximation Property (CAP), $1 \leq p < 2$. We now use them to obtain nontrivial extensions of spaces without the CAP.

The novelty of these examples is in the lack of structure of the spaces in the interpolation scale. Consider, for example, one of the singularity criterions of \cite{Castillo01} (Corollary 5.11): the spaces must have a common $1-$monotone basis, and we must have knowledge of the asymptotic behaviour of successive vectors and of the behaviour of block-sequences.

We follow the presentation of \cite{Lindenstrauss1979} (Theorem 1.g.4).

Let $\sigma_n = \{2^n, ..., 2^{n+1} - 1\}$ and $\Delta_n$ be a suitable partition of $\sigma_n$. For $1 \leq p \leq 2$, consider the following space:
\[
X_p = (\oplus_{n \geq 2} \oplus_{A \in \Delta_n} \ell_2(A))_p
\]

$X_p$ is isomorphic to $\ell_p$ for $1 < p \leq 2$ and $X_1$ is isomorphic to a subspace of $\ell_1$. By including $X_1$ in $X_2$, we have an interpolation scale $(X_2, X_1)$, and it is easy to see that $(X_2, X_1)_{\theta} = X_p$, where $\frac{1}{p} = \frac{1 - \theta}{2} + \theta$.

For each $i \geq 2$, let
\[
z_i = e_{2i} - e_{2i +1} + e_{4i} + e_{4i + 1} + e_{4i + 2} + e_{4i + 3}
\]

Let $w \in span(z_i)_{i\geq 2} = [z_i]_{i\geq 2}$. We define the $\Delta-$support of $w$ to be the set $\{A \in \cup_n\Delta_n : A \cap supp(w) \neq \emptyset\}$. 

If $Y_p$ is the closed subspace of $X_p$ generated by $(z_i)_{i \geq 2}$, we have that $Y_p$ is a space without the CAP, for $1 \leq p < 2$.

Let $W_p = (Y_2, Y_1)_{\theta}$, where $\frac{1}{p} = \frac{1 - \theta}{2} + \theta$. Of course, the first question is whether $W_p = Y_p$. We do not know that. In general, interpolation does not preserve subspaces.

\begin{props}\label{pro:4.3} For each $1 < p < 2$, $W_p$ is a space without the CAP.
\end{props}
\begin{proof}
To see that $W_p$ does not have the CAP, we show that all that is used to prove that $Y_p$ does not have it still holds in $W_p$.

If the $\Delta-$support of $w$ has size $n$, by taking the constant function $w$ we have
\[
\|w\|_{W_p} \leq n^{1 - \frac{1}{p}}\|w\|_{Y_p}
\]

So
\[
\|z_i\|_{W_p} \leq 6^{1 - \frac{1}{p}}\|z_i\|_{Y_p}
\]

\noindent and therefore $\sup \|z_i\|_{W_p} < \infty$.

For $i \geq 2$, let
\[
z_i^* = \frac{e^*_{2i} - e^*_{2i +1}}{2}
\]

Then $z_i^*(z_i) = 1$ for every $i \geq 2$. Since $[z_i]$ is dense in $Y_1$ and $Y_1 \subset Y_2$, it follows from Lemma 4.2.3 of \cite{Bergh01} that $[z_i]$ is dense in $W_p$. We also have that $z_i^*$ is $w^*-$convergent to $0$.

Because of the inclusion $W_p \subset X_p$ (from interpolation), we have $\|z^*\|_{W_p^*} \leq \|z^*\|_{X_p^*}$ for every $z^* \in X_p^*$.

One of the crucial steps of the proof that $Y_p$ does not have the CAP is that we can nicely bound
\[
\|\sum \theta_j y_j\|_{Y_p}
\]

\noindent where $y_j \in Y_p$ are specific vectors appearing in the construction. These sums have $\Delta-$support of size at most $9$, and therefore
\[
\|\sum\theta_j y_j\|_{W_p} \leq 9^{1 - \frac{1}{p}}\|\sum\theta_j y_j\|_{Y_p}
\]

\noindent the factor $9^{1 - \frac{1}{p}}$ being immaterial.

This is enough to reproduce the proof that $Y_p$ does not have the CAP, and thus $W_p$ does not have it either.
\end{proof}

The interpolation scheme $(W_2, W_1)$ induces for each $1 < p < 2$ a twisted sum
  \begin{equation}\label{eq:4.3}
  \xymatrix{ 0 \ar[r] & W_p \ar[r] & dW_p
    \ar[r] & W_p\ar[r]  &0}
  \end{equation}

Let $\Omega_p$ be the quasi-linear map that defines \eqref{eq:4.3}. Using \hyperref[thm:2.14]{Theorem 2.14}, we are able to obtain:
\begin{theorem}\label{thm:4.4}
For each $1 < p < 2$, the quasi-linear map $\Omega_p$ is nontrivial on each subspace of $W_p$ which contains a subsequence $(z_{i(n)})$ of $(z_n)$ for which there is a sequence $(j(n))$ of naturals such that:
\begin{enumerate}
    \item For every $n \in \mathbb{N}$, we have $supp (z_{i(n)}) \subset \sigma_{j(n)} \cup \sigma_{j(n) + 1}$
    \item If $n \neq m$, $\sigma_{j(n)} \neq \sigma_{j(m)}$ and $\sigma_{j(n)} \neq \sigma_{j(m) + 1}$.
\end{enumerate}
\end{theorem}

Notice that we can take $i(n) = 2^{2n-1}$ and $j(n) = 2n$, $n \geq 1$. In particular, \ref{eq:4.3} is nontrivial.
\begin{proof}
Let $W$ be a subspace of $W_p$ like in the enunciate. Notice first that for every $x \in W_p$ we have
\[
\|x\|_{X_p} \leq \|x\|_{W_p}
\]

Let $x \in W$ be a finite sum
\[
x = \sum t_n z_{i(n)}
\]

By considering all the possible contributions of $z_{i(n)}$ to the norm of $x$, we see that, for each $1 \leq p \leq 2$, there is a constant $K_p \geq 1$ independent of $x$ such that
\[
\sum \left|t_n\right|^p \leq \|x\|_{X_p}^p \leq K_p \sum \left|t_n\right|^p
\]

Then, if $\|x\|_{X_p} = 1$, taking the function
\[
f(z) = \sum \left|t_n\right|^{\frac{p}{p(z)}} \frac{t_n}{\left|t_n\right|} z_{i(n)}
\]

where $\frac{1}{p(z)} = \frac{1-z}{2} + z$, we have $f(\theta) = x$ and
\[
\|f(iz)\|_{Y_2} \leq K_2^{\frac{1}{2}} (\sum \left|t_n\right|^p)^{\frac{1}{2}} \leq K_2^{\frac{1}{2}}
\]

and
\[
\|f(1 + iz)\|_{Y_1} \leq K_1 \sum \left|t_n\right|^p \leq K_1
\]

Therefore, $\|x\|_{W_p} \leq K = \max\{K_2^{\frac{1}{2}}, K_1\}$. So, for all $x \in [z_{i(n)}]$, $\|x\|_{X_p} \leq \|x\|_{W_p} \leq K \|x\|_{X_p}$, and we have that the closed subspace generated by $(z_{i(n)})_n$ is isomorphic to a subspace of $X_p$. Combining this with \hyperref[pro:2.2]{Proposition 2.2}, we get that $\frac{1}{p_{W}} = \frac{1 - \theta}{p_{Y_2}} + \frac{\theta}{p_{Y_1}}$, and \hyperref[thm:2.14]{Theorem 2.14} gives the result.
\end{proof}

As a matter of fact, we also have:
\begin{props}\label{pro:4.5}
For $1 < p < 2$, the space $dW_p$ does not have the CAP.
\end{props}
\begin{proof}To see that, we first find a Banach norm on $dW_p$. The twisted sum $dX_{\theta}$ is equivalent to the twisted sum
  \begin{equation*}
  \xymatrix{ 0 \ar[r] & X_{\theta} \ar[r] & \mathcal{F}/(\ker(\delta_{\theta})\cap\ker(\delta'_{\theta}))
    \ar[r] & X_{\theta}\ar[r]  &0}
  \end{equation*}
  
\noindent where we identify $\mathcal{F}/(\ker(\delta_{\theta})\cap\ker(\delta'_{\theta}))$ with
\[
\{(f(\theta), f'(\theta)) : f \in \mathcal{F}\}
\]

\noindent with inclusion $x \mapsto (0, x)$ and quotient $(x, y) \mapsto x$ (\cite{Kalton20031131}, for example).

Let $\tilde{z_i} = (z_i, 0)$. Again
\[
\|\tilde{z_i}\| \leq 6^{1 - \frac{1}{p}}\|z_i\|_{Y_p}
\]

Define $\tilde{z^*_i}$ on $\mathcal{F}/(\ker(\delta_{\theta})\cap\ker(\delta'_{\theta}))$ by
\[
\tilde{z^*_i}(x, y) = z^*_i(x)
\]

Then
\begin{equation*}
\left|\tilde{z^*_i}(x, y)\right| \leq \|z^*_i\|_{W^*_p}\|(x, y)\|
\end{equation*}

The set of points $(w, y)$ with $w \in [z_i]$ is dense in $\mathcal{F}/(\ker(\delta_{\theta})\cap\ker(\delta'_{\theta}))$, since $[z_i]$ is dense in $W_p$, and therefore $\tilde{z^*_i}$ is $w^*-$convergent to $0$.

Instead of finding a bound for $\|\sum\theta_j y_j\|$, we have to bound $\|\sum\theta_j (y_j, 0)\|$. But by taking constant functions
\[
\|\sum\theta_j (y_j, 0)\| \leq 9^{1 - \frac{1}{p}} \|\sum\theta_j y_j\|_{Y_p}
\]

So replacing $z_i$ by $\tilde{z_i}$ and $z_i^*$ by $\tilde{z_i^*}$, we can reproduce the proof that $Y_p$ does not have the CAP to $\mathcal{F}/(\ker(\delta_{\theta})\cap\ker(\delta'_{\theta}))$.
\end{proof}

For more on twisted sums involving the approximation property, see for example \cite{Castillo2013}, \cite{Chen2013}, \cite{Godefroy1989} and \cite{szankowski2009}.

\section{Final Remarks}

One may wonder if the condition on the type of $X_{z_0}$ of \hyperref[thm:2.14]{Theorem 2.14} can be removed. Simple examples show that this is not the case: consider the case $X_0 = \ell_1 \oplus \ell_2$, $X_1 = 0 \oplus \ell_2$. Then $p_{X_0} = 1$, $p_{X_1} = 2$, and for every $\theta \in (0, 1)$ we have $X_{\theta} = \ell_2$ and the trivial extension. It is not clear what happens if we also suppose that $\Delta(\overline{X})$ is dense in both $X_0$ and $X_1$.

Finally, examples \ref{sec:4.1} and \ref{sec:4.2} raise the following

\begin{question}
If $dX_{\theta}$ is nontrivial for every $\theta \in (0, 1)$, $\theta \neq \frac{1}{2}$, is it true that $dX_{\frac{1}{2}}$ is nontrivial?
\end{question}

\section*{Acknowledgements}
This paper is part of my Phd. research at Universidade de S\~ao Paulo under the supervision of Valentin Ferenczi, whom I would like to thank for all his invaluable help. I also would like to thank Jes\'us Castillo for all his kind remarks regarding this work, and Jo\~ao Fernando da Cunha and Wilson Cuellar for their help with some of the background.

\bibliographystyle{amsplain}
\bibliography{refs}

\end{document}